\documentclass{amsart}
\usepackage{amssymb,latexsym}
\usepackage{graphicx}
\newtheorem{theorem}{Theorem}
\newtheorem{lemma}{Lemma}
\newtheorem{definition}{Definition}
\newtheorem{corollary}{Corollary}

\newtheorem{example}{Example}
\newtheorem{conjecture}{Conjecture}

\newtheorem{proposition}{Proposition}
\newtheorem{computer search}{Computer Search}
\begin{document}

\title{The Slide Dimension of Point Processes}
\author{Bill Ralph}

\address{Mathematics Department\\
Brock University\\
St. Catharines, Ontario\\
Canada L2S 3A1\\
\textnormal{Email: bralph@brocku.ca}\\
\textnormal{Phone: (905) 688-5550 x3804}}

\date{April 16, 2014}
\begin{abstract}
We associate with any finite subset of a metric space an infinite sequence of scale invariant numbers $\rho_1,\rho_2,\dots$ derived from a variant of differential entropy called the genial entropy. As statistics for point processes, these numbers often appear to converge in simulations and we give examples where $1/\rho_1$ converges to the Hausdorff dimension. We use the $\rho_n$ to define a new notion of dimension called the slide dimension for a special class of point processes on metric spaces.  The slide calculus is developed to define $\rho_n$ and an explicit formula is derived for the calculation of $\rho_1$.  For a uniform random variable X on $[0,1]^n$, evidence is given that $\rho_1(X) =1/n$ and $\rho_2(X) =-\pi^2/(6n^2)$ and simulations with a normal variable $Z$ suggest that $\rho_1(Z) =4/\pi$ and $\rho_2(Z) =-1$. Some potential applications to spatial statistics are considered.
\end{abstract}
\keywords{genial entropy, slide calculus, slide statistics, Hausdorff dimension, slide dimension, fractal, spatial statistics, metric space, level statistics, fractal analysis, point process, tangible process}

\subjclass{Primary 60G55, Secondary 28A80}

\maketitle
\section{Introduction}\label{S:intro}
The investigation of many important processes in science and mathematics often yields data in the form of a set of points in a metric space that we must somehow quantify and interpret. For example, the field of fractal analysis has developed in order to obtain dimensional information from a wide range of real world datasets using quantities like the Hausdorf, information and correlation dimensions ~\cite{F,H}.  Even more elaborate techniques have been developed, such as the use of the singularity spectrum in multifractal analysis  ~\cite{GP}, in an effort to deal with the general problem of extracting meaningful information from very complex sets of data.  This paper offers a new approach to this problem by introducing a novel sequence of scale invariant numbers $\rho_1,\rho_2,\dots$ called slide statistics that can be computed from any finite set $U$ of distinct elements in a metric space.  When $U$ is taken to be a larger and larger sample of a random variable $X$, we often observe the values of $\rho_n(U)$ approach intriguing limiting values $\rho_n(X)$. For example, $1/\rho_1$ appears to converge to the Hausdorff dimension for many standard fractals. We will propose a new notion of dimension called the slide dimension for a very specific class of point processes called tangible point processes that satisfy a family of constraints on the $\rho_n$.  As we will see, the numbers $\rho_n$ are just the derivatives of a particular function constructed from the data using a variant of differential entropy we call the genial entropy. 

As a first example, consider the uniform random variable $X$ on $[0,1]^n$ for which it appears that $\rho_1(X) =1/n$ and $\rho_2(X) =-\pi^2/(6n^2)$.  In this case,  $1/\rho_1(X)$ is equal to the dimension of the underlying space which may be useful because $1/\rho_1(X)$ is easier to compute than many traditional measures of dimension.  For a normal variable $Z$, the slide statistics appear to converge to $\rho_1(Z) =4/\pi$ and $\rho_2(Z) =-1$.  This last result suggests a simple goodness of fit test for normality based on the closeness of $\rho_2(U)$ to  $-1$. In both of these examples $\rho_1(X) \geq 1$, and we conjecture this relationship to hold for all continuous real-valued random variables. 

Here is a brief overview of how the values of $\rho_n(U)$ are calculated from $U$. In Section~\ref{S:2}, we introduce a variant of the differential entropy called the genial entropy which is the fundamental idea behind all of our results.  Unlike the differential entropy, the genial entropy is scale invariant and in Section~\ref{S:inequality} we prove it is never negative.  Given a finite set of points $U$ in a metric space, we can find the distance from each point to its nearest neighbour and arrange these distances in descending order so $ d_1 \geq d_2 \geq \dots \geq d_n > 0$. Let $f(x)$ be the function on $[0,1)$ whose value on $[(i-1)/n, i/n)$ is $d_i$. Let $A(t)$ be the area under $(f(x))^t$ and let $\sigma(t)$ denote the genial entropy of the density $\frac{(f(x))^t}{A(t)}$ which happens to be $0$ at $t=0$. In Section~\ref{C:calculus}, we develop the slide calculus which then allows us to define $\rho_n(U)$ as the $n$th derivative from the right at $0$ of the function $\sigma(t)$.  

In Section~\ref{C:calculus}, we derive an explict formula for $\rho_1(U)$ and state a conjectured formula for $\rho_2(U)$. The results obtained from the simulation of $\rho_1$ and $\rho_2$ in a variety of contexts are summarized in Section~\ref{S:slide} and Section~\ref{S:assembly} and we identify many interesting relationships. In Section~\ref{S:level}, we introduce an alternative to the slide statistics.  When defining the slide statistics, we use the functions $(f(x))^t$ which can be thought of as a continuous deformation of $f(x)$ at $t=1$ into the constant function $1$ at $t=0$.  We can achieve the same effect using the functions $tf(x)+(1-t)$ so in Section~\ref{S:level} we set up the corresponding derivatives to obtain the level statistics which turn out to be much easier to calculate but often don't converge. They do however have one advantage over the slide statistics in that points in the sample do not have to be distinct.

\section{The genial Entropy}\label{S:2}
Our starting point for the development of the slide statistics is a variant of differential entropy called the genial entropy or g-entropy which will be described in Definition~\ref{D:2.02}. In the simplest case of a probability density $f$ that has an inverse function, the genial entropy is just the sum of the differential entropies of $f$ and $f^{-1}$.  The next theorem shows how this sum can be written in a form that makes sense for densities that may not have inverses.

\begin{proposition}\label{P:2.01} Suppose $f$ is a continuous function on $[0,b]$ with the properties that the derivative of $f$ exists and is negative on $(0,b)$, $f(b)=0$ and $\int_0^b fdx=1$. Also assume the differential entropies of $f$ and  $f^{-1}$ both exist. Then the sum of their differential entropies is given by $ -\int_0^b fln(f)dx-\int_{0}^{f(0)} f^{-1}ln(f^{-1})dy =  -1-\int_0^b fln(xf)dx$. 
\end{proposition}

\begin{proof} 
Substitute $y=f(x)$ into  the integral $-\int_{0}^{f(0)} f^{-1}ln(f^{-1})dy$ to get $-\int_{b}^{0} xln(x)f'(x)dx$ which equals $\int_{0}^{b} xln(x)f'(x)dx$. After integrating by parts this last integral becomes $ -1-\int_0^b fln(x)dx$ and the result follows.
\end{proof}
The genial entropy will only be defined for densities of the following special form.
\begin{definition} \label{D:2.01}  A corner density is a function $f:I \rightarrow [0,\infty)$ where $I \cup \{0\}$ is a connected interval contained in $[0,\infty)$, $f$ is monotone decreasing  and $\int_I fdx=1$. 
\end{definition}
The following definition of genial entropy is one of the central ideas of this paper and is motivated by the conclusion of Proposition~\ref{P:2.01}.
\begin{definition} \label{D:2.02} Let $f:I \rightarrow [0,\infty)$ be a corner density. The genial entropy or g-entropy of $f$ is defined by $G(f)= -1-\int_I fln(xf)dx$ when this integral exists, with the usual convention that $0 \ln (0) = 0$. 
\end{definition}
If $f$ satisfies the conditions of Proposition~\ref{P:2.01}, then $G(f) = G(f^{-1})$ so in particular $-\ln(x)$ and $e^{-x}$ must have the same genial entropy which happens to be Euler's constant as shown in Table~\ref{T:2.01}. In Section~\ref{S:inequality}, we will prove the genial entropy is always nonnegative.

\begin{table}[!htb]{The genial entropy of some corner densities.}
\begin{center}

\begin{tabular}{|c|c|c|}
 \hline
 \textit{Density} & \textit{Domain} &   \textit{Genial Entropy}  \\
\hline
$1/b$ & $[0,b]$ & $0$   \\
\hline
$-\ln(x)$  & $(0,1]$ & $\gamma$ \\
\hline
$e^{-x}$  & $[0,\infty)$ & $\gamma$   \\
\hline
$a/(x^{1-a})$ for $a \in (0,1)$  & $(0,1]$ & $-\ln(a)$  \\
\hline
$2e^{-x^2}/\sqrt{\pi}$  & $[0,\infty)$ & $(-1+\gamma+\ln(\pi))/2 $   \\
\hline
$2/(\pi(1+x^2))$  & $[0,\infty)$ & $-1+\ln(2)+\ln(\pi) $  \\
\hline

\end{tabular}
\caption{ }\label{T:2.01}
\end{center}
\end{table}
Unlike the differential entropy, the genial entropy is invariant under changes of scale.

\begin{theorem}\label{T:2.02} Let $f:I \rightarrow [0,\infty)$ be any corner density.  Then for any $\lambda >0$ the function $h:\lambda I \rightarrow [0,\infty)$ defined by $h(z)= \frac{1}{\lambda} f(\frac{z}{\lambda})$ for $z \in \lambda I$ is a corner density with the same genial entropy as $f$.
\end{theorem}

\begin{proof} 

$G(h) = -1-\int_{\lambda I} hln(zh)dz$ $=$ $-1-\int_{\lambda I} \frac{1}{\lambda} f(\frac{z}{\lambda})ln(z \frac{1}{\lambda} f(\frac{z}{\lambda}))dz$. After substituting $z=\lambda x$, this last integral becomes $-1-\int_I fln(xf)dx = G(f)$.
\end{proof}

In Section~\ref{S:slide}, the following functions will be used to associate a genial entropy with sample data.

\begin{definition}\label{D:2.03} Let $D$ be any sequence $d_1 , d_2 , \dots , d_n $ with $ d_1 \geq d_2 \geq \dots \geq d_n > 0$ and let $D^*$ be the sequence $d_1/\mu , d_2/\mu , \dots , d_n/\mu $ where $\mu$ is the mean of the $d_i$.

\begin{enumerate}
\item Define $f_D:[0,1)\rightarrow[0,\infty)$ to have the value $d_i$ on the interval $[(i-1)/n, i/n)$. In particular, $f_{D^*}$ is the corner density whose value on $[(i-1)/n, i/n)$ is $d_i/\mu$.  
\item Define $L_{D}:[0,\infty)\rightarrow[0,1]$ by $L_{D}(y)= m/n$ where $m$ is the number of elements of $D$ that are less than or equal to $y$. In other words, $L_{D}$ is the restriction of the empirical cumulative distribution function for the data in $D$ to the interval $[0,\infty)$. 
\end{enumerate}
\end{definition}

Part (1) of the next theorem shows $f_D$ is a generalized inverse ~\cite{KMP} of $1-L_D$ that is closely related to the usual quantile function.  When working with sample data, it will be easiest to calculate the genial entropy using $f_{D^*}$ but it can also be calculated using the more familiar empirical cumulative distribution function as shown by part (2) of Theorem~\ref{T:2.03} which is a variation on Proposition~\ref{P:2.01}.

\begin{theorem}\label{T:2.03} Let $D$ be any sequence $d_1 , d_2 , \dots , d_n $ with $ d_1 \geq d_2 \geq \dots \geq d_n > 0$. Then
\begin{enumerate}
\item $f_D(x)=inf \{y \geq 0 | 1 - L_D(y) \leq x \}$ for $x \in [0,1)$
\item  $f_{D^*}$ and $1-L_{D^*}$ are corner densities with the same genial entropy. 
\end{enumerate}
\end{theorem}
\begin{proof} 

(1)	Suppose $x \in [(i-1)/n, i/n)$ for some $i$.  Then $\inf \{y \geq 0 | 1 - L_D(y) \leq x \}= \inf \{y \geq 0 | 1 - L_D(y) < i/n \}= \inf \{y \geq 0 | L_D(y) > 1-i/n \} = \inf \{ d_1, d_2, \dots , d_i \}= d_i = f_D(x)$.

(2)	The sequence $D$ may contain repetitions, so assume that the discontinuities of $f_{D^*}$ in $(0,1)$ occur at $t_i$ with $t_1 < t_2 \dots < t_{m-1}$ and let $t_0=0$ and $t_m=1$. Suppose the value of $f_{D^*}$ on $[t_{i-1},t_i)$ is $e_i$ for $i=1,2,\dots,m$ and let $e_{m+1}=0$ so $ e_1 > e_2 > \dots > e_m > e_{m+1}=0$. We can then describe $1-L_{D^*}$ as the function that takes the value $t_i$ on $[e_{i+1},e_i)$ for $i=1,2,\dots,m$ and the value $0$ on $[e_1,\infty)$. Since $\int_0^1 f_{D^*}(x)dx=\sum_{i=1}^n (d_i/\mu )(i/n-(i-1)/n) =1$, we must also have that $\sum_{i=1}^m e_i(t_i-t_{i-1})=1$. Rearranging this last sum gives $\sum_{i=1}^m t_i(e_i-e_{i+1})=1$ so $\int_0^\infty 1-L_{D^*}(y)dy=1$ and $1-L_{D^*}$ is a corner density.

To see that  $f_{D^*}$ and $1-L_{D^*}$ have the same genial entropy, we use $t_0=0$, $e_{m+1}=0$ and the convention $0 \ln (0) = 0$ to obtain:

\begin{equation}
\begin{split}
-1-\int_0^1 f_{D^*}(x)ln(xf_{D^*}(x))dx
& =-1-\sum_{i=1}^m (\int_{t_{i-1}}^{t_i} e_i ln(x e_i)dx)\\
& =\sum_{i=1}^m ( t_{i-1} e_i  ln(t_{i-1} e_i )-t_i e_i  ln(t_i e_i ))\\
& =\sum_{i=1}^m ( e_{i+1} t_i ln(e_{i+1} t_i))-e_i t_i ln(e_i t_i) \\
& =-1-\sum_{i=1}^m (\int_{e_{i+1}}^{e_i} t_i ln(y t_i)dy)\\
& =-1-\int_0^{\infty} (1-L_{D^*}(y))ln(y(1-L_{D^*}(y)))dy
\end{split}
\end{equation}
\end{proof}

In view of Theorem~\ref{T:2.02}, we could have defined $f_{D^*}$ on any interval $[0,b)$ and obtained a function with the same genial entropy.    The interval $[0,1)$ was chosen in particular to insure the relationship between the genial entropies of  $f_{D^*}$ and $1-L_{D^*}$ stated in part (2) of Theorem~\ref{T:2.03}. We now show that the genial entropy is never negative.

\section{The genial entropy Inequality}\label{S:inequality}
Our goal in this section is to prove that the genial entropy of a corner density can never be negative. The idea is to first prove the necessary inequalities for step functions and then use the fact that monotone functions can be uniformly approximated by step functions.  We begin with an inequality for step functions.

\begin{lemma}\label{L:3.01} Suppose that $y_1 \geq y_2 \geq \dots \geq y_n \geq 0$ and $ 0 =t_0 \leq t_1 \leq t_2 \leq \dots \leq t_n$. Let $A = \sum_{i=1}^n y_i (t_i-t_{i-1})$. Then  $\sum_{i=1}^n (t_{i-1} y_i ln(t_{i-1} y_i)-t_i y_i ln(t_i y_i)) \geq -AlnA $
\end{lemma} 
\begin{proof} Given vectors $a,b \in R^n$, recall from ~\cite{MO} that $a$ majorizes $b$ provided \\ $\sum_{i=1}^k a_{(i)} \geq \sum_{i=1}^k b_{(i)}$ for all $k = 1 \dots (n-1)$ with equality for $k=n$ and where the components of $a$ and $b$ have been sorted in descending order so $a_{(1)} \geq a_{(2)} \geq \dots \geq a_{(n)}$ and $b_{(1)} \geq b_{(2)} \geq \dots \geq b_{(n)}$.  Alternatively, $a$ majorizes $b$ if vector $b$ can be obtained from vector $a$ by a sequence of "transfers" that allow us to change a vector $a = (a_1,a_2, \dots,a_i, \dots,a_j, \dots, a_n)$ into a vector $b = (a_1,a_2, \dots,a_i+\Delta, \dots,a_j -\Delta, \dots, a_n)$ provided $a_i \leq a_j$ and $\Delta \leq a_j-a_i$. 

We now show that the vector $(A, y_2 t_1, y_3 t_2, \dots , y_n t_{n-1})$ majorizes \\ $(y_1 t_1, y_2 t_2, \dots , y_n t_n)$ so the required inequality follows immediately from Karamata's Inequality ~\cite{MO} and the convexity of $x \ln x $ on $[0,\infty)$.  Since $A = y_n (t_n-t_{n-1}) +\sum_{i=1}^{n-1} y_i (t_i-t_{i-1})$ and  $ \sum_{i=1}^{n-1} y_i (t_i-t_{i-1})  \geq y_n\sum_{i=1}^{n-1}  (t_i-t_{i-1}) = y_n t_{n-1}$, we can transfer an amount $\Delta =y_n (t_n-t_{n-1})$ from the first to the last entry of $V_1=(A, y_2 t_1, y_3 t_2, \dots , y_n t_{n-1})$ which means that $V_1$ majorizes $V_2$ $=$  $(A-\Delta, y_2 t_1, y_3 t_2, \dots , y_n t_{n-1} +\Delta)$ $ =$ $ (\sum_{i=1}^{n-1} y_i (t_i-t_{i-1}), y_2 t_1, y_3 t_2, \dots , y_n t_n) $. Now transfer an amount $\Delta =y_{n-1} (t_{n-1}-t_{n-2})$ from the first to the 2nd last entry of $V_2$ and continue similarly with the other entries to obtain the required majorization. 
\end{proof}

\begin{lemma}\label{L:3.02} Suppose that $y_1 \geq y_2 \geq \dots \geq y_n > 0$ and $ 0 =t_0 < a = t_1 \leq t_2 \leq \dots \leq t_n =b$. Let $f$ be the function whose value is $y_i$ on $[t_{i-1},t_i)$ and let $P = \int_0^a f = y_1a$ and $Q = \int_a^b f$.  Then $-\int_a^b f(1+ln(xf))dx\geq PlnP -(P+Q)ln(P+Q)$
\end{lemma} 
\begin{proof} Let $g$ be the constant function whose value is $C > 0$ on $[u,v]$ where $0\leq u<v$.  Then $-\int_u^v g(1+ln(xg))dx = (Cu)ln(Cu)-(Cv) ln(Cv)$ with the usual convention that $0 \ln 0 =0$ in the case when $u=0$. Then $-\int_a^b f(1+ln(xf))dx = -\int_0^b f(1+ln(xf))dx +\int_0^a f(1+ln(xf))dx = \sum_{i=1}^n (y_i t_{i-1}  ln( y_i t_{i-1} )- y_i t_i ln(y_i t_i))+PlnP \geq -(P+Q)ln(P+Q) + PlnP $, after applying Lemma~\ref{L:3.01} with $A=P+Q$.

\end{proof} 

\begin{lemma}\label{L:3.03} Let $f$ be a positive monotone decreasing function on $[a,b]$ where $0 <  a < b$ . Let $P = af(a)$ and $Q = \int_a^b f dx$.  Then $-\int_a^b f(1+ln(xf))dx\geq PlnP -(P+Q)ln(P+Q)$.
\end{lemma} 
\begin{proof} By  ~\cite{SH}, there exists a sequence of monotone decreasing step functions $f_k$ that converge uniformly to $f$ with $0 < f(b) \leq f_k(x) \leq f(a)$ for every $k$ and every $x \in [a,b]$. By standard results, the sequence $ f_k(1+ln(xf_k)$ converges uniformly to  $ f(1+ln(xf)$ on $[a,b]$. Let $P_k = af_k(a)$ and  $Q_k  = \int_a^b f_k dx$.  By  Lemma~\ref{L:3.02}, $-\int_a^b f_k(1+ln(xf_k))dx\geq P_klnP_k -(P_k+Q_k)ln(P_k+Q_k)$ and the result now follows by taking the limit as $k \rightarrow \infty$ of both sides of this inequality.
\end{proof}

\begin{lemma}\label{L:3.04} Let $f$ be a positive monotone decreasing function on $(0,b]$ for which $Q = \int_0^b f dx$ is finite.  Then $-\int_0^b f(1+ln(xf))dx\geq  -QlnQ$.
\end{lemma} 
\begin{proof} Suppose $a$ is between $0$ and $b$ and let $P_a = af(a)$, $Q_a = \int_a^b f dx$. Since $Q$ is finite, $lim_{a\rightarrow 0^+} P_a =0$. By Lemma~\ref{L:3.03},  $-\int_a^b f(1+ln(xf))dx\geq P_alnP_a -(P_a+Q_a)ln(P_a+Q_a)$ and the result follows by taking the limit as $a\rightarrow 0^+$ on both sides of this inequality. 
\end{proof}

\begin{lemma}\label{L:3.05} Let $f$ be a positive monotone decreasing function on $(0,\infty)$ for which $Q = \int_0^{\infty} f dx$ is finite.  Then $-\int_0^{\infty} f(1+ln(xf))dx \geq  -QlnQ$.
\end{lemma} 
\begin{proof}  Let $Q_t = \int_0^t f dx$ for $t > 0$,. By Lemma~\ref{L:3.04},  $-\int_0^t f(1+ln(xf))dx \geq - Q_t ln(Q_t)$ and the result follows by taking the limit as $t\rightarrow \infty$ on both sides of this inequality. 
\end{proof}
\begin{theorem}\label{T:3.01} (The Genial Entropy Inequality) For any corner density $f:I \rightarrow [0,\infty)$, $-1-\int_I f(ln(xf)dx\geq 0$.
\end{theorem} 
\begin{proof} Follows from Lemma~\ref{L:3.05} by taking $Q=1$.
\end{proof}
If the density function of a random variable happens to be a corner density, then the Genial Entropy Inequality gives a lower bound for the differential entropy.
\begin{theorem}\label{T:3.02} Let $X$ be a random variable whose pdf is a corner density $f:I \rightarrow [0,\infty)$ and let $h(X)=-\int_I fln(f)dx$ be the differential entropy of $X$. Then $h(X) \geq 1+ E(\ln X)$.
\end{theorem} 
\begin{proof} Follows from Theorem~\ref{T:3.01}.
\end{proof}

\section{The Slide Calculus}\label{C:calculus}

With each corner density, we now associate a function called a slide function that describes how the genial entropy changes as the density is deformed to a constant function. In Section~\ref{S:slide}, the slide numbers will be defined as the derivatives of a particular slide functions at $0$. 

\begin{definition}\label{D:4.00} Suppose  $f:(0,b) \rightarrow (0,\infty)$ is monotone decreasing and $A(t) =\int_0^b (f(x))^t dx$.Then  $\sigma_f(t) = G\big(\frac{(f(x))^t}{A(t)} \big) = -1 -\int_0^b \frac{f(x))^t}{A(t)} \ln(x\frac{f(x))^t}{A(t)})dx$ will be called the slide function of $f$ and its domain will be the set of all $t \geq 0$ at which $A$ and $\sigma_f$ both exist.
\end{definition}

By the Genial Entropy Inequality in Theorem~\ref{T:3.01}, we always have $\sigma_f(t)\geq 0$. Note also that $\sigma_f(0)=0$ and if $\int_0^b f=1$ then $\sigma_f(1) = G(f)$. The next theorem says that the function $\sigma_f(t)$ is invariant under changes of scale. It follows from a simple change of variables argument similar to the proof of Theorem~\ref{T:2.03}. 

\begin{theorem}\label{T:4.01} Suppose  $f:(0,b) \rightarrow (0,\infty)$ is monotone decreasing and  $g:(0,\lambda b) \rightarrow (0,\infty)$ is defined by $g(z)=\beta f(\frac{z}{\lambda})$ for some  $\lambda >0$ and any $\beta >0$. Then $\sigma_f=\sigma_g$. 
\end{theorem}

We now show that under mild conditions there must be an $s>0$ for which the interval $[0,s]$ is contained in the domain of $\sigma_f$ and furthermore that $\sigma_f$ must be continuous from the right at $0$.

\begin{lemma}\label{L:4.01} Suppose  $f:(0,b) \rightarrow (0,\infty)$ is monotone decreasing with $\int_0^b (f(x))^s dx < \infty$ for some $s >0$. Then 
$\sigma_f$ is defined on $[0,s/2]$ and continuous from the right at $0$.
\end{lemma} 
\begin{proof} 

We can assume $b=1$ by Theorem~\ref{T:4.01}. Then for $t \in [0,s]$, the function $A(t)$ in Definition~\ref{D:4.00} is finite since $A(t) =\int_0^1 (f(x))^t dx \leq \int_0^1 1+(f(x))^s dx =1+ A(s) < \infty$. To see that $lim_{t\rightarrow 0^+} A(t) =1$, choose $c$ and $d$ with $0 < c < d < 1$ and consider $|A(t)-1|=|\int_0^1 (f(x))^t -1 dx | \leq \int_0^1 |(f(x))^t -1 |dx \leq \int_0^c  (1+f(x)^s) dx + \int_c^d |(f(x))^t -1| dx +\int_d^1 (1+f(x)^s) dx$. In this last sum, the first and third terms are independent of $t$ and can be made as small as desired by choosing $c$ and $d$ appropriately.  Since $|(f(x))^t -1 |$ converges uniformly to $0$ on $[c,d]$ as $t\rightarrow 0^+$, the results follow.

We now show that $\sigma_f(t)$ is finite for $t \in [0,s/2]$ and continuous from the right at $0$ as follows: 
\begin{equation}\notag
\begin{split}
0&\leq \sigma_f(t)\\
&= -1 -\int_0^1 \frac{(f(x))^t}{A(t)} \ln(x\frac{(f(x))^t}{A(t)})dx\\
&=-1 - \frac{1}{A(t)} \int_0^1 (f(x))^t ln(x) dx- \frac{1}{A(t)} \int_0^1 (f(x))^t ln((f(x))^t) dx \\
&+ \frac{1}{A(t)} \int_0^1 (f(x))^t \ln(A(t)) dx \\
&=\frac{1}{A(t)} \big(1-A(t)+A(t)\ln (A(t)) -\int_0^1 ((f(x))^t-1) ln(x) dx \\
&-\int_0^1 (f(x))^t ln((f(x))^t) dx \big)\\
&\leq \frac{1}{A(t)} \big(|1-A(t)+A(t)\ln (A(t)| +\int_0^1 |(f(x))^t-1|| \ln(x)| dx \\
&+\int_0^1 |(f(x))^t \ln((f(x))^t)| dx \big)
\end{split}
\end{equation}

It remains to show that each of the three terms in this last sum is finite for $t \in [0,s/2]$ and goes to $0$ as $t\rightarrow 0^+$. For the first term, clearly \\  $\lim_{t\rightarrow 0^+} \big(1-A(t)+A(t)\ln (A(t) \big)=0$.

For the second term, $\int_0^1 |(f(x))^t-1|| \ln(x)|dx$ converges since the integrand is the product of square integrable functions. To see that $\lim_{t\rightarrow 0^+} \int_0^1 |(f(x))^t-1|| \ln(x)| dx =0$, choose $c$ and $d$ with $0 < c < d < 1$ and consider the inequality
$\int_0^1 |(f(x))^t-1|| \ln(x)| dx \leq \int_0^c (1+f(x)^{s/2})| \ln(x)| dx+\int_c^d |(f(x))^t-1|| \ln(x)| dx+\int_d^1 (1+f(x)^{s/2})| \ln(x)| dx$. Now follow the argument given above for $|A(t)-1|$.

For the third term, use the inequality $z-1 \leq z\ln z \leq z(z-1)$ for $z \geq 0$ to get  $ \int_0^1 |(f(x))^t \ln((f(x))^t)| dx \leq \int_0^1 |(f(x))^t-1| \max(1,(f(x))^t) dx $ which converges since the last integrand is the product of square integrable functions. Now show $lim_{t\rightarrow 0^+} \int_0^1 |(f(x))^t-1| \max(1 , (f(x))^t) dx =0$ as before. 

\end{proof}

The information we wish to extract from a corner density is captured by the derivatives of its slide function at $0$ that we now describe.

\begin{definition} \label{D:4.01} Suppose  $f:(0,b) \rightarrow (0,\infty)$ is monotone decreasing with $\int_0^b (f(x))^s dx < \infty$ for some  $s>0$. Then the $n$'th slide derivative of $f$ is defined by $\psi_n(f)=\frac{d^n \sigma_{f}}{dt^n}(0)$ where all derivatives are taken from the right. If all of these derivatives exist, then the slide series of $f$ is defined to be $\sum_{i=1}^{\infty} \frac{\psi_n(f)}{n!}t^n$.

\end{definition}

Here are some elementary properties of $\psi_n(f)$.

\begin{theorem}\label{T:4.02}  Suppose  $f:(0,b) \rightarrow (0,\infty)$ is monotone decreasing with $\int_0^b (f(x))^s dx < \infty$ for some  $s>0$.
\begin{enumerate}

\item If $\psi_1(f)$ exists then $\psi_1(f) \geq 0$.

\item If $\psi_n(f)$ exists then so does $\psi_n(f^r) $ for $r>0$ and $\psi_n(f^r) =r^n \psi_n(f)  $.  

\item If $f$ is a constant function, then $\psi_n(f) =0$ for all n.
\end{enumerate}
\end{theorem} 
\begin{proof} 

(1) The corner density $\sigma_f$ is nonnegative on its domain and $\sigma_f(0)=0$ so the first derivative must be nonnegative.

(2) For $t$ sufficiently small and nonnegative, we have $\sigma_{f^r}(t)=\sigma_f(rt)$ and the result follows from the chain rule.  

(3) If $f$ is a constant function, then $\sigma_f(t)=0$ for all $t \geq 0$.

\end{proof}

Here is an example of a slide derivative calculation that we will connect with the uniform distribution on $[0,1]$.

\begin{theorem}\label{T:4.03} Let $f(x)=-\ln(x) $ for $x \in (0,1)$. Then for $t > 0$,  $\sigma_f(t)=-1+t-t\Psi(t)+\log(\Gamma(1+t))$ and the slide derivatives of $f$ are given by $\psi_1(f)=1$ and by  $\psi_n(f)=(-1)^{n+1}(n-1)!(n-1)\zeta(n)$ for $n>1$ .
\end{theorem} 
\begin{proof} 

Let $g_t(x)=\frac{(f(x))^t}{A(t)}=\frac{(-\ln(x))^t }{\Gamma(1+t)}$ so $\int_0^1 g_t(x)dx =1$. Then
\begin{equation}\label{xx}
\begin{split}
\sigma_f(t)&= G(g_t(x))\\
&= -1 -\int_0^1g_t\ln(xg_t)dx\\
& = -1 -\int_0^1g_t\ln(x)dx -\int_0^1g_t\ln(g_t)dx\\
 &= -1 -(-1-t)-(1-\ln(\Gamma(1+t))+t\Psi(t))\\
  &= -1 +t+\ln(\Gamma(1+t))-t\Psi(t)
\end{split}
\end{equation}

The result now follows by differentiating $-1 +t+\ln(\Gamma(1+t))-t\Psi(t)$ from the right at $t=0$.
\end{proof}

\begin{corollary}\label{C:4.01} Let $f_r(x)=(-\ln(x))^r$ for $x \in (0,1)$ and $r>0$. Then $\psi_1(f_r)=r$ and $\psi_n(f_r)=(-1)^{n+1}(n-1)!(n-1)\zeta(n)r^n$ for $n>1$. The slide series for $f_r$ is given by $rt+  \sum_{n=2}^{\infty}\frac{(-1)^{n+1} (n-1)\zeta(n) r^n t^n}{n!}$.
\end{corollary} 

\begin{proof} 
Immediate from Theorem~\ref{T:4.02} and Theorem~\ref{T:4.03}.

\end{proof}

\begin{example} \label{E:4.01} $\psi_1(f)=\infty$ for the function $f(x)=exp(-1/(1-x)^2)$  on $[0,1)$. 
\end{example}

As mentioned in the introduction, we will consider a set of points in a metric space and find the distance $d_i$ from the $i$'th point to its nearest neighbour.  These distances can then be used to construct the function $f_D$ in Definition~\ref{D:2.03}. The next theorem gives an explicit formula for the first slide derivative $\psi_1(f_D)$ that will be central to the next section. This theorem also demonstrates that the slide derivatives of a function can exist even when the function is not continuous.

\begin{theorem}\label{T:4.04} Suppose  $ D=\{d_1 ,d_2, \dots, d_n \} $ where we assume $d_1 \geq d_2 \geq \dots \geq d_n > 0$ and let $f_D$ be the function on $[0,1)$ whose value on the interval $[\frac{i-1}{n},\frac{ i}{n})$ is $d_i$ as in Definition~\ref{D:2.03}.  Then the first slide derivative of $f_D$ is given by $\psi_1(f_D)= \frac{1}{n}\sum_{i=2}^{n-1}i\ln(i) \ln(\frac{d_{i+1}}{d_i})+\frac{\ln(n)}{n}\sum_{i=1}^{n-1} \ln(\frac{d_i}{d_n})$.
\end{theorem} 
\begin{proof} 
By Definition~\ref{D:4.01}, we have to calculate the right hand derivative of the slide function $\sigma_{f_D}(t)$ at $t=0$ so we first find an expression for $\sigma_{f_D}(t)$.  Let $g_t(x)=\frac{(f_D(x))^t}{A(t)}$ where $A(t) = \int_0^1 (f_D(x))^t dx$ so $g_t(x)$ takes the value $a_i(t)$ on $[\frac{i-1}{n},\frac{ i}{n})$ where $a_i(t)=\frac{nd_i^t}{\sum_{i=1}^n d_i^t}$.

\begin{equation}\notag
\begin{split}
\sigma_f(t)&=G(g_t(x))\\
&=-1-\int_0^1 g_t(x)\ln(xg_t(x))dx\\
&=-1-\sum_{i=1}^n \int_{\frac{i-1}{n}}^{\frac{i}{n}} g_t(x)\ln(xg_t(x))dx\\
&=-1-\big(\big(\frac{a_1}{n}\big)\ln \big(\frac{a_1}{n}\big)-\big(\frac{a_1}{n}\big)+\sum_{i=2}^n \big(\big(\frac{ia_i}{n}\big)\ln\big(\frac{ia_i}{n}\big)-\big(\frac{(i-1)a_i}{n}\big)\ln\big(\frac{(i-1)a_i}{n}\big)-\frac{a_i}{n}\big)\big)\\
&=-\big(\frac{a_1}{n}\big)\ln \big(\frac{a_1}{n}\big)-\sum_{i=2}^n \big(\frac{ia_i}{n}\big)\ln\big(\frac{ia_i}{n}\big)+\sum_{i=2}^n \big(\frac{(i-1)a_i}{n}\big)\ln\big(\frac{(i-1)a_i}{n}\big)
\end{split}
\end{equation}
To find the derivative of $\sigma_f(t)$, we now use the facts that $\frac{d a_k}{dt} = \frac{n-1}{n}\ln(d_k)-\frac{1}{n}\sum_{i\neq k}^n \ln(d_i)$ and $\frac{d(v \ln v)}{dt}=(1+\ln v)\frac{dv}{dt}$. 
\begin{equation}\notag
\begin{split}\frac{d \sigma_f}{dt}(0)&= \frac{d}{dt} \big( \big(\frac{a_1}{n}\big)\ln \big(\frac{a_1}{n}\big)-\sum_{i=2}^n \big(\frac{ia_i}{n}\big)\ln\big(\frac{ia_i}{n}\big)+\sum_{i=2}^n \big(\frac{(i-1)a_i}{n}\big)\ln\big(\frac{(i-1)a_i}{n}\big) \big)_{t=0}\\
&=-\big(1+\ln \big(\frac{a_1(0)}{n}\big)\big) \big( \frac{1}{n}\big)\big( \big(\frac{n-1}{n}\big)\ln(d_1)-\frac{1}{n}\sum_{j\neq 1}^n \ln(d_j) \big)\\
&-\sum _{i=2}^{n}\big(\big(1+\ln \big(\frac{ia_i(0)}{n}\big)\big) \big( \frac{i}{n}\big)\big( \big(\frac{n-1}{n}\big)\ln(d_i)-\frac{1}{n}\sum_{j\neq i}^n \ln(d_j) \big)\big)\\
&+ \sum _{i=2}^{n}\big(\big(1+\ln \big(\frac{(i-1)a_i(0)}{n}\big)\big) \big( \frac{i-1}{n}\big)\big( \big(\frac{n-1}{n}\big)\ln(d_i)-\frac{1}{n}\sum_{j\neq i}^n \ln(d_j) \big)\big)
\end{split}
\end{equation}
At $t=0$, each $a_i$ is equal to $1$ so this derivative becomes
\begin{equation}\notag
\begin{split}
\frac{d \sigma_f}{dt}(0)&= -\big(1+\ln \big(\frac{1}{n}\big)\big) \big( \frac{1}{n}\big)\big( \big(\frac{n-1}{n}\big)\ln(d_1)-\frac{1}{n}\sum_{j\neq 1}^n \ln(d_j) \big)\\
&-\sum _{i=2}^{n}\big(\big(1+\ln \big(\frac{i}{n}\big)\big) \big( \frac{i}{n}\big)\big( \big(\frac{n-1}{n}\big)\ln(d_i)-\frac{1}{n}\sum_{j\neq i}^n \ln(d_j) \big)\big)\\
&+\sum _{i=2}^{n}\big(\big(1+\ln \big(\frac{i-1}{n}\big)\big) \big( \frac{i-1}{n}\big)\big( \big(\frac{n-1}{n}\big)\ln(d_i)-\frac{1}{n}\sum_{j\neq i}^n \ln(d_j) \big)\big)\\
&=-\big(1+\ln \big(\frac{1}{n}\big)\big) \big( \frac{1}{n}\big)\big( \big(\frac{n-1}{n}\big)\ln(d_1)-\frac{1}{n}\sum_{j\neq 1}^n \ln(d_j) \big)\\
&+\sum _{i=2}^{n}\big(\big(-1+(i-1)\ln(i-1)-i\ln(i) +\ln(n)\big) \big( \frac{1}{n}\big)\big( \big(\frac{n-1}{n}\big)\ln(d_i)-\frac{1}{n}\sum_{j\neq i}^n \ln(d_j) \big)\big)\\
&=P_1+P_2+P_3
\end{split}
\end{equation}
The $P_i$ terms are defined and calculated as follows.
The $P_1$ term consists of the parts of this last expression that involve the isolated $1$s.
\begin{equation}\notag
\begin{split}
P_1&= -\big(1\big) \big( \frac{1}{n}\big)\big( \big(\frac{n-1}{n}\big)\ln(d_1)-\frac{1}{n}\sum_{j\neq 1}^n \ln(d_j) \big)\\
& +\sum _{i=2}^{n}\big(\big(-1\big) \big( \frac{1}{n}\big)\big( \big(\frac{n-1}{n}\big)\ln(d_i)-\frac{1}{n}\sum_{j\neq i}^n \ln(d_j)
 \big)\big)\\
 &=0
\end{split}
\end{equation}
$P_2$ is the sum of all of the terms containing $\ln n$.
\begin{equation}\notag
\begin{split}
P_2&= \big( \frac{\ln n}{n^2}\big)\big( \big(n-1\big)\ln(d_1)-\sum_{j\neq 1}^n \ln(d_j) \big)\\
& +\big( \frac{\ln n}{n^2}\big)\sum _{i=2}^{n}\big( \big( \big(n-1\big)\ln(d_i)-\sum_{j\neq i}^n \ln(d_j)
 \big)\big)\\
 &+\big( \frac{-n\ln n}{n^2}\big)\big( \big(n-1\big)\ln(d_n)-\sum_{j\neq n}^n \ln(d_j) \big)\\
 &= \big( \frac{-n\ln n}{n^2}\big)\big( \big(n-1\big)\ln(d_n)-\sum_{j\neq n}^n \ln(d_j) \big)\\
&=\frac{\ln(n)}{n}\sum_{i=1}^{n-1} \ln(\frac{d_i}{d_n})
\end{split}
\end{equation}
$P_3$ is what remains after $P_1$ and  $P_2$ are subtracted from $\frac{d \sigma_f}{dt}(0)$.
\begin{equation}\notag
\begin{split}
P_3&= \sum _{i=2}^{n}\big(\big((i-1)\ln(i-1)-i\ln(i)\big) \big( \frac{1}{n}\big)\big( \big(\frac{n-1}{n}\big)\ln(d_i)-\frac{1}{n}\sum_{j\neq i}^n \ln(d_j) \big)\big)\\
&+\big( \frac{n\ln n}{n^2}\big)\big( \big(n-1\big)\ln(d_n)-\sum_{j\neq n}^n \ln(d_j) \big)\\
&=\frac{1}{n^2} \sum_{i=2}^{n-1} i \ln i \big( \big(-(n-1) \ln(d_i)+\sum_{j\neq i}^n \ln (d_j) \big) +\big((n-1) \ln(d_{i+1})-\sum_{j\neq {i+1}}^n \ln (d_j)   \big)  \big)\\
&=\frac{1}{n^2} \sum_{i=2}^{n-1} i \ln i \big( -n\ln(d_i)+n \ln(d_{i+1}) \big)\\
&= \frac{1}{n}\sum_{i=2}^{n-1}i\ln(i) \ln(\frac{d_{i+1}}{d_i})
\end{split}
\end{equation}
\end{proof}

The following formula for the second slide derivative $\psi_2(f)$ is motivated by calculations for small values of $n$. In the next section, we will see that results from simulations based on this formula agree with what we would expect from theoretical considerations. 

\begin{conjecture}\label{CJ:4.01} Suppose  $ d_1 \geq d_2 \geq \dots \geq d_n > 0$ and let $f_D$ be the function on $[0,1)$ whose value on the interval $[\frac{i-1}{n},\frac{ i}{n})$ is $d_i$. Let $S_1=\sum_{i=1}^n \log(d_i)$, $S_2=\sum_{i=1}^n \log(d_i)^2$ and $S_3=\sum_{i=1}^{n-1} log(d_i/d_n)^2$. Then the second slide derivative of $f_D$ is given by 
\begin{multline*}
\psi_2(f_D)= - \big( \sum_{i=1}^{n-1} \big( i \log(i) \log(d_{i+1}/d_i)(2S_1-n \log(  d_i d_{i+1}))\big)\\
+\log(n) (2(S_1-n \log(d_n))^2-nS_3) +nS_2 - S_1^2 \big)/n^2.
\end{multline*}
\end{conjecture}

The next two sections define the slide numbers and assembly numbers and demonstrates their application to some standard point processes. 

\section{The Slide Numbers}\label{S:slide}
Sample data often consists of a set of distinct points in a metric space.  With the help of Definition~\ref{D:4.01}, we can now associate with each of these samples an infinite family of new statistics called the slide numbers. 

\begin{definition}\label{D:5.01} Let M be a metric space and let $U= \{u_1, u_2, \dots, u_k \}$ be a set of k distinct points in $M$. For each $i=1, \dots, k$, let $d_i$ be the distance from $u_i$ to its nearest neighbour in $U$.  Define a sequence $D$ by ordering the $d_i$ in descending order as $ d_{[1]} \geq d_{[2]} \geq \dots \geq d_{[k]} >0$. As in Definition~\ref{D:2.03}, let $f_{D}$ be the function on $[0,1)$ whose value on the interval $[\frac{i-1}{k},\frac{ i}{k})$ is $d_{[i]}$. Define the $n$'th slide number of $U$ by $\rho_n(U) = \psi_n(f_{D})$ and define the slide series of $U$ to be $\sum_{i=1}^{\infty} \frac{\rho_n(U)}{n!}t^n$.
\end{definition}

Values of $\rho_1(U)$ for various random variables are shown in Table~\ref{T:5.01} which shows the connection between $1/\rho_1(U)$ and the Hausdorff dimension of $[0,1]^m$, the Cantor set and the Sierpinski triangle. For certain point processes P, the statistics $\rho_n(U)$ appear to converge as the sample size $U$ gets large.  For example, for any normal random variable $Z$, the quantity $\rho_1(U)$ appears to converge to $4/\pi$ so it makes sense to define $\rho_1(Z)=4/\pi$ and more generally to define $\rho_n(P)$ for an arbitrary point process $P$ as follows.

\begin{definition}\label{D:5.03} Let $M$ be a metric space and let $U = \{u_1, u_2, \dots \}$ be a sequence of distince points in $M$ generated by some point process $P$ and let $ U_k = \{u_1, u_2, \dots, u_k \}$.  If $\rho_n(U_k)$ converges in probability as $k \rightarrow \infty$, then $\rho_n(P)$ is defined to be the value of this limit. If all of the limits $\rho_n(P)$ exist, then we define the slide series of the process $P$ to be $\sum_{i=1}^{\infty} \frac{\rho_n(P)}{n!}t^n$. In the case where $U$ is a sample of a random variable $X$, we will use the notation $\rho_n(X)$ instead of $\rho_n(P)$.
\end{definition}

Some evidence for the convergence of the slide statistics $\rho_n(U)$ is given in  Table~\ref{T:5.01} and Table~\ref{T:5.02}. In these tables, the points in the Cantor set were generated using $\sum_{i=1}^{40} \frac{a_i}{3^i}$  where the $a_i$ are either $0$ or $2$ with probability $1/2$. Points in the Sierpinski triangle were generated using the Chaos Game ~\cite{BM}. The generation of all random numbers used in these simulations was based on The Mersenne Twister. In the case of a random variable $X$ for which $\rho_n(X)$ exists, note that $\rho_n(aX+b)=\rho_n(X)$ for all real number $a$ and $b$ with $a \neq 0$.  In particular, adjusting the mean or standard deviation of a random variable has no effect on $\rho_n(X)$.

\begin{table}[!htb]{Simulated Values of $\rho_1(U)$ for Various Random Variables}
\begin{center}
\begin{tabular}{|c|c|c|c|c|}
 \hline
\textit{Density} & $\mu_1$ &  $\sigma_1$ & $1/\mu_1$  & $\frac{1}{\rho_1}$ \textit{conjectured}    \\
\hline
uniform on $[a,b]$ & $1.0003$ & $0.0111$  & $0.9997$  & $1$   \\
\hline
{normal} & $1.2664$ & $0.129 $  & $0.7896$ & $\pi/4 \simeq 0.785 $  \\
\hline
exponential  & $1.4590$ & $0.0141$ & $0.6854$  & ?  \\
\hline
$1/(2\sqrt{x})$ on $[0,1]$  & $1.2817$ & $0.0132 $  & $0.7802$ & ?  \\
\hline
uniform on $[0,1]^2$ & $0.5023$ & $0.0056$  & $1.9908$ & $2$   \\
\hline
uniform on $[0,1]^3$ & $0.3416$  & $0.0037$  & $2.9274$ & $3$   \\
\hline
uniform on $[0,1]^4$ & $0.2642$ & $0.0029$  & $3.7850$ & $4$   \\
\hline
bivariate normal & $0.7264$ & $0.0073 $  & $1.3766$ & ? \\
\hline
Cantor  & $1.6014$ & $0.0170 $  & $0.6244$ & $\ln(2)/\ln(3)  \approx 0.631$   \\
\hline
Sierpinski  & $0.6344$ & $0.0067 $  & $1.57624$ & $\ln(3)/\ln(2)  \approx 1.5849$   \\
\hline

\end{tabular}
\caption{ For each density, $1000$ samples of size $10000$ were generated and the value of $\rho_1(U)$ was computed for each sample $U$. The mean $\mu_1$ and standard deviation $\sigma_1$ of these $1000$ values for $\rho_1(U)$ are shown. The value given for $1/\rho_1$ is the conjectured limiting value of $1/\mu_1$ as the sample size approaches infinity.}\label{T:5.01}
\end{center}
\end{table}

\begin{table}[!htb]{Simulated Values of $\rho_2(U)$ for Various Random Variables}
\begin{center}

\begin{tabular}{|c|c|c|c|c|}
 \hline
\textit{Density} & $\mu_2$ &  $\sigma_2$  &  $\rho_2$ \textit{conjectured}   \\
\hline
uniform on $[a,b]$ & $-1.6461$ & $.0732$   & $-\pi^2/6 \approx -1.6449$   \\
\hline
normal & $-1.0273$ & $0.0860 $   & 1  \\
\hline
exponential & $-0.7333$ & $0.0920 $   & ?  \\
\hline
$1/(2\sqrt{x})$ on $[0,1]$  & $-2.5792$ & $0.1085$  & ? \\
\hline
uniform on $[0,1]^2$ & $-0.4096$ & $0.0186$   & $-(\pi^2/6)(1/2)^2 \approx -.4112$   \\
\hline
uniform on $[0,1]^3$ & $-0.1825$  & $0.0083$   & $-(\pi^2/6)(1/3)^2 \approx -0.1827$   \\
\hline
uniform on $[0,1]^4$ & $-0.1038$ & $0.0049$   & $-(\pi^2/6)(1/4)^2 \approx 0.1028$   \\
\hline
bivariate normal & $-0.2004$ & $0.0233 $   & ?  \\
\hline
Cantor & $-4.1464$ & $0.1933 $   &  $ \frac{(-1)^{2+1}(2-1)!(2-1)\zeta(2)}{(ln(2)/ln(3))^2} \approx -4.132$   \\ 
\hline
Sierpinski & $-0.6549$ & $0.0295 $   &  $ \frac{(-1)^{2+1}(2-1)!(2-1)\zeta(2)}{(ln(3)/ln(2))^2} \approx -0.655$   \\ 
\hline

\end{tabular}
\caption{ For each density, $1000$ samples of size $10000$ were generated and the value of $\rho_2(U)$ was computed for each sample $U$ using   Conjecture~\ref{CJ:4.01}. The mean $\mu_2$ and standard deviation $\sigma_2$ of these $1000$ values for $\rho_2(U)$ are shown. The value given for $\rho_2$ is the conjectured limiting value of $\mu_2$ as the sample size approaches infinity. }\label{T:5.02}
\end{center}
\end{table}

Consider the following outline for a possible argument to explain the empirical results obtained for $[0,1]^m$ in Table~\ref{T:5.01} and Table~\ref{T:5.02}.  Suppose  $U = \{u_1, u_2, \dots, u_k \}$ is a very large number of points chosen uniformly at random from  $[0,1]^m$ and let $x$ be a particular point in $U$. By \cite{SKM}, the probability that a point is within $r$ of $x$ is approximately $1-e^{-\alpha r^m}$ for some $\alpha$. If the sample $U$ is large enough, then the set of nearest neighbour distances will be sufficiently independent \cite{K,S} that their empirical cumulative distribution will also be approximately equal to $1-e^{-\alpha r^m}$. If $D$ is the ordered sequence of nearest neighbour distances, then $L_{D^*}(r)$ in Definition~\ref{D:2.03} will be approximately $1-e^{-\beta r^m}$ for some $\beta$ and $1-L_{D^*}(r)$ will be approximately $e^{-\beta r^m}$.  By part (1) of Theorem~\ref{T:2.03}, $f_{D^*}(x)$ is a generalized inverse of $1-L_{D^*}(r)$ so $f_{D^*}(x)$ should be approximately $\frac{(-\log(x))^\frac{1}{m}}{\Gamma(1+\frac{1}{m})}$. Corollary~\ref{C:4.01} now suggests the following conjecture which is supported by the empirical results shown in Table~\ref{T:5.01} and Table~\ref{T:5.02}.

\begin{conjecture}\label{CJ:5.01} Let $U = (u_1, u_2, \dots )$ be a sequence of points chosen uniformly at random from $[0,1]^m$ and let $ U_k = \{u_1, u_2, \dots, u_k \}$. Then as $k \rightarrow \infty$,  $\rho_1(U_k)$ converges in probability to $1/m$ and  $\rho_n(U_k)$ converges in probability to $(-1)^{n+1}(n-1)!(n-1)\zeta(n)/m^n$ for $n >1$.
\end{conjecture} 

There appear to be cases other than $[0,1]^m$ for which the dimension equals $m$ and $\rho_n$ converges to $(-1)^{n+1}(n-1)!(n-1)\zeta(n)/m^n$.  For example, if we take $m=\log(2)/\log(3)$ then $(-1)^{2+1}(2-1)!(2-1)\zeta(2)/m^2 \approx -4.132$ which is close to the value shown in Table~\ref{T:5.02} for the  Cantor set. A similar result holds for the Sierpinski triangle and prompts us to make the following definition. 

\begin{definition}\label{D:5.04} In the context of Definition~\ref{D:5.03}, we say that a point process $P$ is tangible provided there is a number $d$ with $\rho_1(P)=1/d$  and $\rho_n(P)=(-1)^{n+1}(n-1)!(n-1)\zeta(n)/d^n$ for $n>1$. The number $d$ will be called the slide dimension of the process.  If there is no such number, the process will be called intangible. 
\end{definition}

According to the values for $\rho_1$ and $\rho_2$ given in Table~\ref{T:5.01} and Table~\ref{T:5.02}, the normal distribution does not satisfy the conditions for tangibility in  Definition~\ref{D:5.04} so cannot be assigned a slide dimension. Also, if we substitute the estimates for $\rho_2$ from Table~\ref{T:5.02} into the inverse function $\frac{\pi}{\sqrt{-6\rho_2}}$, we get considerably better estimates for the dimension of $[0,1]^m$ than the values for $1/\mu_1$ shown in Table~\ref{T:5.01}. For a tangible process $P$, we would like to know in general if the statistics $\sqrt[n]{\frac{(-1)^{n+1}(n-1)!(n-1)\zeta(n)}{\rho_n}}$ converge more quickly to the dimension for larger values of $n$.

For the subset $U$ of $R$ generated by $20000$ iterations of $x_{i+1}=x_i+cos(i)$ with $x_0=0$, the value of $\rho_1(U)$ is approximately $0.53$. But values of $\rho_1(U)$ less than $1$ cannot occur for continuous real-valued random variables according to the next conjecture which is supported by the results in Table~\ref{T:5.01}. 

\begin{conjecture}\label{CJ:5.03} If $X$ is a continuous real-valued random variable for which $\rho_1(X)$ exists, then $\rho_1(X) \geq 1$.
\end{conjecture} 

The values of $\rho_2(X)$ are all negative in Table~\ref{T:5.02} and we know from Theorem~\ref{T:4.02} that $\rho_1(X) \geq 0$ so there is some support for the following conjecture.

\begin{conjecture}\label{CJ:5.04} If $X$ is a continuous real-valued random variable for which $\rho_n(X)$ exists, then $ (-1)^{n+1}\rho_n(X) \geq 0$.
\end{conjecture}

\section{The Assembly Numbers}\label{S:assembly}
 
In analyzing the spatial characteristics of a set of points, it is sometimes preferrable to use all of the interpoint distances \cite{BF,S} rather than just the distances to nearest neighbours as we have done so far. The assembly numbers, that we now define, are like the slide numbers from Definition~\ref{D:5.01} except that we use the distances between every pair of points.

\begin{definition}\label{D:6.01} Let M be a metric space and let $U= \{u_1, u_2, \dots, u_k \}$ be a set of k distinct points in $M$. Let $D$ be the sequence ${d_1 , d_2 , \dots , d_m}$ of distances between each pair of points in $U$ where we assume that $ d_1 \geq d_2 \geq \dots \geq d_m > 0$. Let $f_{D}$ be the function on $[0,1)$ whose value on the interval $[\frac{i-1}{m},\frac{ i}{m})$ is $d_i$  as in Definition~\ref{D:2.03}. Define the nth assembly number of $U$ by $\alpha_n(U) = \psi_n(f_{D})$ and define the assembly series of $U$ to be $\sum_{i=1}^{\infty} \frac{\alpha_n(U)}{n!}t^n$.
\end{definition}

Table~\ref{T:6.01} illustrates that the values of $\alpha_1(U)$ are remarkably stable even for very small samples. In general, it appears that $\alpha_1(U)$ converges much faster than $\rho_1(U)$ which might make  $\alpha_1(U)$ a better spatial statistic for small samples.

\begin{table}[!htb]{Simulated Values of $\alpha_1(U)$ for Samples in $[0,1]^m$}
\begin{center}
\begin{tabular}{|c|c|c|c|c|c|}
 \hline
\textit{ Sample Size} & $[0,1]$ & $[0,1]^2$ &  $[0,1]^3$ &  $[0,1]^4$   \\
\hline
$10$  & $0.7607 \pm 0.0880$ & $0.4415 \pm 0.0465$ & $0.3310 \pm 0.0363$   & $0.2746 \pm 0.0313$    \\
\hline
$20$   & $0.7785 \pm 0.0476$ & $0.4533 \pm 0.0238$ & $0.3410 \pm 0.0197$   & $0.2827 \pm 0.0172$    \\
\hline
$50$   & $0.7856 \pm 0.0209$ & $0.4596 \pm 0.0112$ & $0.3458 \pm 0.0091$   & $0.2868 \pm 0.0085$   \\
\hline
$100$   & $0.7883 \pm 0.0116$ & $0.4612 \pm 0.0061$ & $0.3474 \pm 0.0056$  & $0.2880 \pm 0.0055$  \\
\hline

\end{tabular}
\caption{ The mean and standard deviation of the statistic $\alpha_1(U)$ for $1000$ samples from uniform distributions on $[0,1]^m$. }\label{T:6.01}
\end{center}
\end{table}

Table~\ref{T:6.02} shows some of the intriguing values obtained for the first and second assembly numbers from simulations. These results may prove useful in various goodness of fit tests particularly because the standard deviations shown here are so small.

\begin{table}[!htb]{Simulated values of $\alpha_1(U)$ and $\alpha_2(U)$ for various random variables}
\begin{center}
\begin{tabular}{|c|c|c|c|c|}
 \hline
 \textit{Random Variable } & $\mu_1$ & $\mu_2$ &  $1/\mu_1$  \\
\hline
$u$  & $0.7897\pm 0.0023 \approx \pi /4$  & $-1.5589\pm0.0190$ & $1.2662$     \\
\hline
 {Normal} & $0.8113\pm 0.0049 \approx 8/\pi^2 $  & $-1.4395\pm0.0074$ & $1.2325$     \\
\hline
$\log(u)$  & $0.9987\pm 0.0166 \approx 1$ & $-1.6491\pm0.0201 \approx -\pi^2 /6$ & $1.0013$    \\
\hline
$( \sin(2\pi u), \cos(2\pi u))$  & $0.5205 \pm 0.0009 \approx \pi /6$ & $-1.2739 \pm 0.0064 \approx -4/\pi $ & $1.9210 $    \\
\hline
$(u \sin(2\pi v),u \cos(2\pi v))$  & $0.5078 \pm 0.0058$ & $-0.5073 \pm 0.0172$ & $1.9692 $    \\
\hline
$(u,v)$  & $0.4624 \pm 0.0011$ & $-0.4525\pm0.0058$ & $2.1626 $    \\
\hline
$(u,v,w)$  & $0.3482 \pm 0.0013$ & $-0.2262\pm0.0022$ & $2.8715 $    \\
\hline
$\{(u,v)|u^2+v^2 \leq1\}$  & $0.4548 \pm 0.0006$ & $-0.4497\pm0.0056$ & $2.1987 $    \\
\hline
{Bivariate Normal}  & $0.4998 \pm 0.0041  \approx 1/2$ & $-0.4113\pm0.0030$ & $2.0008 $    \\
\hline

\end{tabular}
\caption{ The average values $\mu_1$ and  $\mu_2$ of the statistics $\alpha_1(U)$ and $\alpha_2(U)$. The random variables $u$,$v$ and $w$ are independent and uniform in $[0,1]$. In each case $1000$ samples of size $1000$ were used. }\label{T:6.02}
\end{center}
\end{table}

There is an alternative approach to the slide calculus and the slide statistics that we now consider briefly. 
\section{The Level Calculus}\label{S:level}

Here is a variation on Definition~\ref{D:4.01} that accomplishes the deformation of a corner density to a constant function in a linear way.

\begin{definition} \label{D:7.01} Let $f:(0,b) \rightarrow (0,\infty)$ be a corner density. For $t \in [0,1]$, define the level function of $f$ by $\lambda_f(t) =G(tf(x)+(1-t)/b)  = -1 -\int_0^b (tf(x)+(1-t)/b)\ln(x(tf(x)+(1-t)/b))dx$, for values of $t$ at which this integral exists.
\end{definition}

The level derivatives are now defined in a similar way to the slide derivatives in Definition~\ref{D:4.01}.

\begin{definition} \label{D:7.02} Let $f:(0,b) \rightarrow (0,\infty)$ be a corner density. The $n$'th level derivative of $f$ is defined by $\lambda_n(f)=\frac{d^n \lambda_{f}}{dt^n}(0)$ where all derivatives are taken from the right. If all of these derivatives exist then the level series of $f$ is defined to be $\sum_{i=1}^{\infty} \frac{\lambda_n(f)}{n!}t^n$.
\end{definition}

Unlike the situation for the slide derivatives,  all of the level derivatives of corner densities of the form $f_{D^*}$ in Definition~\ref{D:2.03} are easily described.  The next theorem follows from a routine calculation of the derivatives in Definition~\ref{D:7.02}.

\begin{theorem}\label{T:7.01} Suppose  $ D=(d_1 ,d_2, \dots, d_n ) $ is a sequence of points where we assume $d_1 \geq d_2 \geq \dots \geq d_n \geq 0$ and the $d_i$ are not all zero. Let $\mu$ be the average of the $d_i$ and let $f$ be the corner density on $[0,1)$ whose value on the interval $[\frac{i-1}{n},\frac{ i}{n})$ is $d_i/\mu$.  Then the first level derivative of $f$ is given by $\lambda_1(f)=\sum_{i=1}^{n}(1-d_i/\mu) \big(  (\frac{ i}{n})\ln(\frac{ i}{n})-(\frac{ i-1}{n})\ln(\frac{ i-1}{n}) \big) = -1- \int_0^1 f(x) \log(x) dx = -1 - E(log(x))$. For $n>1$, the level derivatives are given by $\lambda_n(f)=-\frac{ 1}{n}\sum_{i=1}^{n}(1-d_i/\mu)^n  = - \int_0^1 (1-f(x))^n  dx $. 
\end{theorem} 

In this theorem, we note that $\lambda_2(f)$ is just the negative of the square of the coefficient of variation. In parallel with Definition~\ref{D:5.01}, we now associate level statistics with any finite set in a metric space.

\begin{definition}\label{D:7.03} Let M be a metric space and let $U= (u_1, u_2, \dots, u_k )$ be a sequence of k points in $M$. For each $i=1, \dots, k$, let $d_i$ be the distance from $u_i$ to its nearest neighbour in $U$.  Define a sequence $D$ by ordering the $d_i$ in descending order as $ d_{[1]} \geq d_{[2]} \geq \dots \geq d_{[k]}$ and let $\mu$ be the mean of the $d_i$ which we assume are not all zero. As in Definition~\ref{D:2.03}, let $f_{D^*}$ be the corner density on $[0,1)$ whose value on the interval $[\frac{i-1}{k},\frac{ i}{k})$ is $d_{[i]}/\mu$. Define the nth level number of $U$ by $\rho_n^L(U) = \lambda_n(f_{D^*})$ and define the level series of $U$ to be $\sum_{i=1}^{\infty} \frac{\rho_n^L(U)}{n!}t^n$.
\end{definition}

Our definition of the level statistics $\rho_n^L(U)$ does not require points in the sequence $U$ to be distinct which would appear to be an advantage over the slide statistics $\rho_n(U)$.  The problem with the level statistics is that for many standard point processes $\rho_n^L(U)$ appears to either converge very slowly or not at all which is the case for the normal distribution. On the other hand, the slide statistics $\rho_n(U)$ appear to have good convergence properties as shown by the examples in Table~\ref{T:5.01}.  One place in which the statistics $\rho_n^L(U)$ appear to behave reasonably well is for a uniform distribution on $[0,1]^m$ that we now consider. 

In the case of samples $U$ taken uniformly at random from $[0,1]$, the values of $\rho_n^L(U)$ obtained from simulations have very large variances but average out to be roughly equal to $1$, $-1$, $2$, $-9$ and $44$ for $n$ from $1$ to $5$ which suggests a connection with derangements of an $n$ element set. The argument preceding Conjecture~\ref{CJ:5.01} suggests that in this case the limiting values of $\rho_n^L$ might be obtained by substituting the function $f(x) = -\ln(x)$ into the integrals given in Theorem~\ref{T:7.01}. For large samples $U$, the value of $\rho_1^L(U)$ should then be close to $ -1- \int_0^1 (-\ln(x)) \ln(x) dx =1$. For $n >1 $, $\rho_n^L(U)$ should be close to $- \int_0^1 (1-(-\ln(x)))^n  dx  = (-1)^{n+1} \int_0^1 (-\ln(x)-1)^n  dx $. By a change of variables, $\int_0^1 (-\ln(x)-1)^n  dx $ becomes $\int_0^{\infty} (t-1)^n \exp(-t) dt $  which is shown in \cite{KM} to be the number of derangements of an $n$ element set so there may be some value in these ad hoc arguments. More generally, if $f$ is the corner density given by $f(x) = \frac{(-\ln(x))^m}{\Gamma(1+m)}$, then  Corollary~\ref{C:4.01} says $\sigma_1(f) = m$  and Definition~\ref{D:7.02} gives $\lambda_1(f) = -1- \int_0^1 f(x) \ln(x) dx =m$.  Since $\sigma_1(f)=\lambda_1(f)$, we might expect $\rho_1(U)$ and $\rho_1^L(U)$ to both converge to $m$ for a uniform distribution on $[0,1]^m$ which is what we have observed in simulations for small values of $m$.

\section{Conclusions and Future Work}\label{S:conclusion}

As we have seen, the statistics $\rho_n$ are capable of exposing exotic characteristics of finite subsets of a metric space.  Data in this form is so common in science and mathematics that there is the possibility of these unusual statistics finding widespread application in many different fields. 
The simulations we have described show that $\rho_1$ and $\rho_2$ are capable of distinguishing between probability distributions and of detecting dimensional information. In general however, the $\rho_n$ are generally quite mysterious and much work will need to be done to understand what they are telling us about sets of points in metric spaces, random variables or point processes in general.  

What this work is missing is the formal theory to explain the results from our simulations. A good first step might be a proof that for a uniform random variable $X$ on $[a,b]$, $\rho_1(X) =1$ and $\rho_n(X)=(-1)^{n+1}(n-1)!(n-1)\zeta(n)$  for $n>1$.  The rationale preceding Conjecture~\ref{CJ:5.01} might provide a possible outline for this proof but we do not expect any argument concerning convergence in probability to be straightforward especially with statistics as complex as the $\rho_n$. In regard to a normal variable $Z$,  simulations suggest that  $\rho_1(Z) = 4/\pi$ and $\rho_2(Z) = 1$ but we have no theoretical framework to suggest why these should hold and we can't even guess the values of $\rho_n(Z)$ for $n>2$. Ultimately, we would like to have tools for calculating  $\rho_n(X)$ for all of the commonly used random variables.

In this paper, we gave examples of point processes on the Cantor set and the Sierpinski triangle for which $1/\rho_1$ converged to the dimension of the fractal. We would like to understand when this occurs in general and also the relationship between $1/\rho_1$ and the usual definitions of dimension.  More generally, we would like to know when a point process is tangible in the sense of  Definition~\ref{D:5.04}. When a process is tangible it is possible to use the expression $\sqrt[n]{\frac{(-1)^{n+1}(n-1)!(n-1)\zeta(n)}{\rho_n(U)}}$ as a statistic for estimating the dimension. We gave examples in which it converged to the dimension faster for $n=2$ than for $n=1$ and we would like to know what happens for larger values of $n$.

In terms of calculations, we need to prove Conjecture~\ref{CJ:4.01} concerning the calculation of $\rho_2(U)$ and we need formulas for $\rho_n(U)$ for $n>2$.  Given the complexity of our conjecture for $\rho_2(U)$, the formulas for larger values of $n$ are likely to be very complicated. The convergence of the $\rho_n(U)$ for real-valued random variables can sometimes be improved by using the distances between consecutive points rather than the distances to nearest neighbours. We would like to have a better understanding of this situation and also to know if there are any higher dimensional analogues.

There is a natural way to apply slide statistics to kernal density estimation that we will discuss elsewhere ~\cite{R} and applications to the analysis of financial information are under development. Finally, if $U$ consists of the first $20,000,000$ primes, then the value of $\rho_1(U)$ is approximately $0.77235$ which is interesting in view of Conjecture~\ref{CJ:5.03}.


\begin{thebibliography}{11}

\bibitem{BM} 
Barnsley, M.,
\emph{Fractals Everywhere},
Academic Press, 1988.

\bibitem{B} 
Billingsley, P.,
\emph{Probability and Measure},
John Wiley and Sons Ltd, 1995.

\bibitem{BF} 
Bonetti, M., Forsberg, L., Ozonoff, A. and Pagano, M.,
\emph{The distribution of
interpoint distances. Mathematical Modeling Applications in Homeland
Security},
HT Banks and C Castillo Chavez, Eds., 2003,87-106.

\bibitem{E} 
Edgar, G.,
\emph{Measure, Topology, and Fractal Geometry},
Springer-Verlag, 1990.

\bibitem{F} 
Falconer, K.,
\emph{Fractal Geometry: Mathematical Foundations and Applications},
John Wiley and Sons, Ltd., 2003.

\bibitem{GP} 
Grassberger, P. P., Procaccia, I.,
\emph{Measuring the strangeness of strange attractors},
Physica, 9D, 1983,  189-208.


\bibitem{H} 
Harte, D.,
\emph{Multifractals},
London: Chapman and Hall, 2001.

\bibitem{KM} Kayll, M.,
\emph {Integrals Don't Have Anything to Do with Discrete Math, Do They?},
Mathematics Magazine 84(2);2011, 108-119

\bibitem{K} 
Kester, A.,
\emph{Asymptotic Normality of the Number of
Small Sistances between Random Points in a Cube},
Stoch. Proc. Appl. 3,1975, 45-54.

\bibitem{KMP} 
Klement, E. P., Mesiar, R., Pap, E.,
\emph{Quasi- and pseudo-inverses of monotone
functions, and the construction of t-norms},
Fuzzy Sets and Systems, 104, 1999, 3-13.


\bibitem{MO} 
Marshall, A., Olkin, I.,
\emph{Inequalities: Theory of Majorization and Its Applications},
Academic Press, 1980.

\bibitem{R} 
Ralph, B.,
\emph{Kernel Density Estimation Using Slide Statistics}
In Preparation.



\bibitem{S} 
Shurygin, A.,
\emph{Using interpoint distances for pattern recognition},
Pattern Recognition and Image Analysis,16(4),2006,726-729. 


\bibitem{SH} Sohrab, H.,
\emph {Basic Real Analysis},
Birkhauser Boston, 2003.


\bibitem{SKM} 
Stoyan, D., Kendall, W., Mecke, J., 
\emph{Stochastic Geometry and its Applications},
John Wiley and Sons Ltd, 2008.

\end{thebibliography}
\end{document}